\documentclass[11pt]{amsart}
\usepackage[utf8]{inputenc}
\usepackage{enumerate}
\usepackage{bbold}

\title{Gaussian measures of dilations of convex rotationally symmetric sets in $\C^n$}
\author{Tomasz Tkocz}
\address{College of Inter-Faculty Individual Studies in Mathematics and Natural Sciences and\\ Institute of Mathematics \\ University of Warsaw \\ Banacha 2 \\ 02-097 Warszawa \\ Poland}
\email{t.tkocz@students.mimuw.edu.pl}
\keywords{Gaussian measure, convex bodies, isoperimetric inequalities}
\subjclass[2010]{Primary 60E15; Secondary 60G15}

\newtheorem{thm}{Theorem}
\newtheorem{lm}{Lemma}
\newtheorem{cor}{Corollary}

\theoremstyle{definition}
\newtheorem{remark}{Remark}
	
\newcommand{\C}{\mathbb{C}}
\newcommand{\R}{\mathbb{R}}
\newcommand{\I}{\mathbb{1}}
\newcommand{\dd}{\mathrm{d}}
\newcommand{\s}[2]{\langle\!#1,#2\!\rangle}
\newcommand{\fun}[3]{#1\colon #2 \longrightarrow #3}

\begin{document}

\begin{abstract}
We consider the complex case of the so-called \emph{S-inequality}. It concerns the behaviour of the Gaussian measures of dilations of convex and rotationally symmetric sets in $\C^n$ (rotational symmetry is invariance under the transformation $z \mapsto e^{it}z$, for any real $t$). We pose and discuss a conjecture that among all such sets the measure of cylinders (i.e. the sets $\{z \in \C^n \ | \ |z_1| \leq p\}$) decrease the fastest under dilations. 

Our main result of the paper is that this conjecture holds under the additional assumption that the Gaussian measure of considered sets is not greater than some constant $c > 0.64$.
\end{abstract}

\maketitle

\section*{Introduction}\label{sec:intro}

Let us consider the standard Gaussian measure $\nu_n$ on $\C^n$, i.e.
\[ \nu_n(B) = \frac{1}{\left (2\pi\right )^n}\int_{j(B)} \exp\left (-\sum_{k=1}^n(x_k^2+y_k^2)\right )\dd x_1\dd y_1 \ldots \dd x_n \dd y_n, \]
for any Borel set $B \subset \C^n$, where $\fun{j}{\C^n}{\R^{2n}}$ is the standard isomorphism $j((x_1+iy_1, \ldots, x_n+iy_n)) = (x_1, y_1, \ldots, x_n, y_n)$. Denote for any $z = (z_1, \ldots, z_n), w = (w_1, \ldots, w_n) \in \C^n$ by $\s{w}{z} = \sum_{k=1}^n w_k\bar{z_k}$ a scalar product on $\C^n$ and the norm generated by it as $\|z\| = \sqrt{\s{z}{z}}$.

Let $A \subset \C^n$ be a set, which is
\begin{itemize}
	\item convex,
	\item \textbf{rotationally symmetric}, i.e. for any $\lambda \in \C$, $|\lambda| = 1$, $a \in A$ implies that $\lambda a \in A$
\end{itemize}
and $P = \{z \in \C^n \ | \ |\s{z}{v}| \leq p\}$ be a \textbf{cylinder} such that $\nu_n(A) = \nu_n(P)$, where $v \in \C^n$ has  length 1 and $p \geq 0$ is a \textbf{radius} of $P$. We ask whether
\[ \nu_n(tA) \geq \nu_n(tP), \qquad \textrm{for $t\geq 1$}, \]
i.e. the measure of dilations of cylinders grows the slowest among all convex rotationally symmetric sets.

The analogous question in $\R^n$ has an affirmative answer which was shown by R. Lata{\l}a and K. Oleszkiewicz \cite{latole}. Following their method in the considered complex case we obtain a partial answer to the posted question. The main result is the following
\begin{thm}\label{thm1}
There exists a constant $c > 0.64$ such that for any convex rotationally symmetric set $A \subset \C^n$, with measure $\nu_n(A) \leq c$, and a cylinder $P = \{z \in \C^n \ | \ |z_1| \leq p\}$ satisfying $\nu_n(A) = \nu_n(P)$, we have
\begin{equation}\label{eq1}\tag{$*$}
\nu_n(tA) \leq \nu_n(tP), \qquad \textrm{for $0 \leq t \leq 1$}.
\end{equation}
\end{thm}

The paper is organized as follows. In Section \ref{sec:proof} we give the proof of the above theorem. Next, in Section \ref{sec:remarks}, some remarks concerning this theorem are stated. Especially, we discuss the possibility of omitting the restriction on measure assumed in Theorem \ref{thm1}, but weakening its assertion. Section \ref{sec:lemmas} is devoted to proofs of some auxiliary lemmas which have slightly technical character.

\section{Proof of the main result}\label{sec:proof}

Firstly, let us set up some notation. We put $|x| = \sqrt{x_1^2 + \ldots + x_n^2}$ for the standard norm of a vector $x = (x_1, \ldots, x_n) \in \R^n$. By $\gamma_n$ we denote the standard Gaussian measure in $\R^n$ and by $\gamma_n^+(A) := \underline{\lim}_{h \to 0+} (\gamma_n(A^h)-\gamma_n(A))/h$ --- Gaussian perimeter of $A \subset \R^n$, where $A^h := \{x \in \R^n \ | \ \textrm{dist}(x, A) \leq h\}$ is $h$-neighbourhood of $A$. Analogously, we define $\nu_n^+(A)$. Moreover, we will use functions
\begin{align*}
\Phi(x) &= \gamma_1((-\infty, x)) = \frac{1}{\sqrt{2\pi}}\int_{-\infty}^x e^{-t^2/2}\dd t, \\
T(x) &= 1 - \Phi(x).
\end{align*}

Conducting the same procedure as in the real case, presented in detail in \cite{latole}, we can reduce a proof of \eqref{eq1} to some kind of an isoperimetric problem in $\R^3$. However, we loose too much and a constraint involving a boundedness of the measure from above by $c$ appears. For the sake of the reader's convenience, that reduction is briefly presented below. 
\begin{enumerate}[(I)]
	\item For any measurable set $A \subset \C^n$ let $\nu_A(t) := \nu_n(tA)$. Then Theorem \ref{thm1} is equivalent to $\nu_{A}'(1) \geq \nu_{P}'(1)$, provided that $\nu_A(1) = \nu_P(1) \leq c$. Since $P$ is a cylinder we have $\nu_P'(1) = p\nu_n^+(P).$
	\item Convexity of $A$ gives $\nu_A'(1) \geq w\nu_n^+(A)$, where 
\[ w := \sup \{r \geq 0 \ | \ \{z \in \C^n \ | \ \|z\| < r\} \subset A\}. \]
The parameter $2w$ is in some sense the width of the set $A$.
	\item Rotational symmetry of $A$ gives that $A$ is included in some cylinder of the radius $w$. Without loss of generality we may assume that this cylinder is along the first axis, that is $A \subset \{z \in \C^n \ | \ |z_1| \leq w\}.$ Now we can apply Ehrhard's symmetrization \cite{Ehrhard} and obtain a set in $\R^3$
\[ \widetilde{A} = \{(x, y, t) \in \R^3 \ | \ t \leq f\left (\sqrt{x^2+y^2}\right ), \sqrt{x^2+y^2} \leq w\}, \]
where $\fun{f}{[0, w]}{\R\cup \{-\infty\}}$,
\begin{multline*}
f\left (\sqrt{x^2+y^2}\right ) \\:= \Phi^{-1}\left (\nu_{n-1}\left (\{(z_2, \ldots, z_n) \in \C^{n-1} \ | \ (x+iy, z_2, \ldots, z_n) \in A\}\right )\right ),
\end{multline*}
is a well defined (by the rotational symmetry of $A$) concave nonincreasing function (by the convexity of $A$ and Ehrhard's inequality \cite{Ehrhard}). Clearly, $\nu_n(A) = \gamma_3(\widetilde{A})$. The key property of this symmetrization is that $\nu_n^+(A) \geq \gamma_3^+(\widetilde{A})$. Obviously a symmetrized cylinder $P$ is a \textbf{cylinder} $\widetilde{P} = \{z \in \R^2 \ | \ |z| \leq p\}\times\R$ and $\nu_n^+(P) = pe^{-p^2/2} = \gamma_3^+(\widetilde{P})$.
\end{enumerate}
Summing up, in order to prove Theorem \ref{thm1} it is enough to show
\begin{thm}\label{thm2}
There exists a constant $c > 0.64$ with the following property. Let $A \subset \R^3$ be a set of the form
\[ A = \{(x, y, t) \in \R^3 \ | \ t \leq f\left (\sqrt{x^2+y^2}\right ), \sqrt{x^2+y^2} < w\}, \] 
where $\fun{f}{[0, w)}{\R}$ is some concave, nonincreasing, smooth function such that $f(x) \xrightarrow[x\to w-]{} -\infty$. Let $P = \{(x, y, t) \in \R^3 \ | \ \sqrt{x^2+y^2} \leq p\} \subset \R^3$ be a cylinder with the same measure as $A$, that is, $\gamma_3(A) = \gamma_3(P) = 1-e^{-p^2/2}$. Then
\begin{equation}\label{eq:thm2}
w\gamma_3^+(A) \geq p\gamma_3^+(P),
\end{equation}
provided that $\gamma_3(A) \leq c$.
\end{thm}
\begin{proof}
For fixed $x \in [0, w]$ let us define (this idea comes from \cite{latole})
\begin{align*}
A(x) &= A \cup \{z \in \R^2 \ | \ |z| < x\}\times\R, \\
P(x) &= \{z \in \R^2 \ | \ |z| < a(x)\}\times\R,
\end{align*}
where the function $a(x)$ is defined by the equation
\[ \gamma_3(A(x)) = \gamma_3(P(x)).\]
We have $\partial A(x) = B_1(x) \cup B_2(x)$, where $B_1(x) = \{(z, t) \in \R^2\times\R \ | \ |z| = x, t \geq f\left (|z|\right )\}$, $B_2(x) = \{(z, t) \in \R^2\times\R \ | \ |z| > x, t = f\left (|z|\right )\}$. Let
\[ L(x) = w\gamma_3^+(B_2(x)) + x\gamma_3^+(B_1(x)) - a(x)\gamma_3^+(P(x)),  \qquad x \in [0, w]. \]
Since $A(w)$ is a cylinder with the radius $w$, we have $L(w) = 0$. Also note that $L(0) = w\gamma_3^+(A) - p\gamma_3^+(P).$ Therefore it suffices to prove that $L$ is nonincreasing. 

We can easily calculate indispensable things to obtain $L'(x)$. Namely
\begin{align*}
\gamma_3^+(B_2(x)) &= \frac{1}{\sqrt{2\pi}}\int_x^w t\exp\left (-\frac{t^2+f(t)^2}{2}\right )\sqrt{1+f'(t)^2} \dd t, \\
\gamma_3^+(B_1(x)) &= \frac{1}{\sqrt{2\pi}^3}\int_0^{2\pi}\int_{f(x)}^\infty \exp\left (-\frac{x^2+t^2}{2}\right ) x\ \dd t\dd \phi \\
&= xe^{-x^2/2}(1-\Phi(f(x))) = xe^{-x^2/2}T(f(x)),\\
\gamma_3^+(A(x)) &= a(x)e^{-a(x)^2/2}.
\end{align*}
Putting these into the definition of $L$ we have
\begin{align*}
L(x) &= \frac{w}{\sqrt{2\pi}}\int_x^w t\exp\left (-\frac{t^2+f(t)^2}{2}\right )\sqrt{1+f'(t)^2} \dd t + x^2e^{-x^2/2}T(f(x)) \\
&- a(x)^2e^{-a(x)^2/2}.
\end{align*}
Moreover
\begin{align*} \gamma_3(A(x)) &= \gamma_3\left (\{z \in \R^2 \ | \ |z| < x\}\times\R\right ) \\
&+ \gamma_3\left (\{(z, t) \in \R^2\times\R \ | \ |z| > x, t \leq f(|z|)\}\right ) \\
&= 1-e^{-x^2/2} + \int_x^w te^{-t^2/2}\Phi(f(t))\dd t.
\end{align*}
Thus
\[ 1 - e^{-a(x)^2/2} = \gamma_3(P(x)) = \gamma_3(A(x)) = 1-e^{-x^2/2} + \int_x^w te^{-t^2/2}\Phi(f(t))\dd t,\]
and differentiating in $x$ we get
\[ a'(x)a(x)e^{-a(x)^2/2} = xe^{-x^2/2}(1-\Phi(f(x))) = xe^{-x^2/2}T(f(x)).\]
It allows us to compute $L'$. We have
\begin{align*}
L'(x) &= -\frac{w}{\sqrt{2\pi}} x\exp\left (-\frac{x^2+f(x)^2}{2}\right )\sqrt{1+f'(x)^2} \\
&+ e^{-x^2/2}\left (2xT(f(x)) - x^2\frac{e^{-f(x)^2/2}}{\sqrt{2\pi}}f'(x)-x^3T(f(x))\right ) \\
&- \left (2-a(x)^2\right )xe^{-x^2/2}T(f(x)).
\end{align*}
Simplifying a bit one gets that $L' \leq 0$ iff
\[ w\sqrt{1+f'(x)^2} + xf'(x) \geq (a(x)^2-x^2)\sqrt{2\pi}e^{f(x)^2/2}T(f(x)), \qquad x \in [0, w]. \]
Since $f' \leq 0$ ($f$ is nonincreasing) and $\inf_{t \leq 0} (w\sqrt{1+t^2}+xt) = \sqrt{w^2-x^2}$ we will have $L' \leq 0$ if we show that
\begin{equation}\label{eq2}
\sqrt{w^2-x^2} \geq (a(x)^2-x^2)\sqrt{2\pi}e^{f(x)^2/2}T(f(x)), \qquad x \in [0, w].
\end{equation}

Estimating $a(x)^2-x^2$ we can prove the above inequality in some special cases. Notice that monotonicity of $f$ implies $A(x) \subset \{z \in \R^2 \ | \ |z| < x\}\times\R \cup \{(z, t) \in \R^2\times\R \ | \ x \leq |z| \leq w, t \leq f(x)\}$, hence
\[ 1-e^{-a(x)^2/2} = \gamma_3(A(x)) \leq (1-e^{-x^2/2}) + (e^{-x^2/2}-e^{-w^2/2})\Phi(f(x)), \]
so
\begin{equation}\label{eq3}
a(x)^2 - x^2 \leq -2\ln\left (T(f(x)) + \Phi(f(x))e^{-(w^2-x^2)/2}\right ).
\end{equation}
According to it, in order to establish \eqref{eq2} it is enough to show
\[ \sqrt{w^2-x^2} \geq -2\sqrt{2\pi}e^{f(x)^2/2}T(f(x))\ln\left (T(f(x)) + \Phi(f(x))e^{-(w^2-x^2)/2}\right ). \]
In general the above inequality is not true. Nonetheless, there holds Lemma \ref{lm3}, the proof of which is deferred to the last section.

Let us introduce functions $\fun{F}{\R}{(0, \infty)}$, $\fun{G}{(0,\infty)}{(0,\infty)}$ given by formulas
\begin{align}\label{eq:defF}
F(y) &= -\sqrt{2\pi}e^{y^2/2}T(y)\ln T(y),\\
G(y) &= \frac{y}{2(1-e^{-y^2/2})}.\label{eq:defG}
\end{align}
Note that $F$ is increasing and onto (cf. Lemma \ref{lm1}). We will need the constant
\[ H = F^{-1}\left (G\left (\sqrt{8/\pi}\right )\right ). \]
\begin{lm}\label{lm3}
Let either
\begin{enumerate}[(i)]
\item $u \leq \sqrt{8/\pi}$, $y \in \R$, or
\item $u > \sqrt{8/\pi}$, $y \leq H$.
\end{enumerate}
Then
\[ -2\sqrt{2\pi}e^{y^2/2}T(y)\ln\left (T(y) + \Phi(y)e^{-u^2/2}\right ) \leq u.\]
\end{lm}
Applying it for $u = \sqrt{w^2-x^2}$, $y = f(x)$, the desired inequality is true for $x$ such that $\sqrt{w^2-x^2} \leq \sqrt{8/\pi}$ or $\sqrt{w^2-x^2} > \sqrt{8/\pi}$ and $f(x) \leq H$.

Therefore, it remains to prove \eqref{eq2} for $x$ satisfying $\sqrt{w^2-x^2} > \sqrt{8/\pi}$ and $f(x) > H$. Observe that 
\[ (a(x)^2-x^2)' = 2(a(x)a'(x)-x) = 2x\left (e^{(a(x)^2-x^2)/2}T(f(x)) - 1\right ), \]
but thanks to \eqref{eq3} we get
\[ e^{(a(x)^2-x^2)/2} < 1/T(f(x)), \]
hence
\[ (a(x)^2-x^2)' < 0. \]
Thus the function $[0,w] \ni x \longmapsto a(x)^2-x^2 \in [0,\infty)$ is decreasing. It yields
\[ \sup_{x\in[0,w]}(a(x)^2-x^2) = a(0)^2 = p^2.\]

Moreover, the function $x\longmapsto e^{f(x)^2/2}T(f(x))$ is nondecreasing on the interval $\{x\in [0,w] \ | \ f(x) > 0\}$ as a composition of the nonincreasing function $f$ and the decreasing one $y \longmapsto e^{y^2/2}T(y)$ for $y > 0$ (\cite{latole}, Lemma 1). Consequently
\[ \sup \left \{ e^{f(x)^2/2}T(f(x)) \ | \ f(x) > H\right \} = e^{H^2/2}T(H). \]
Combining these two observations and using the assumption $c \geq \gamma_3(A) = \gamma_3(P) = 1-e^{-p^2/2}$, that is $p^2 \leq -2\ln (1-c)$, we obtain that \eqref{eq2} holds for $x$ such that $\sqrt{w^2-x^2} > \sqrt{8/\pi}$ and $f(x) > H$. Indeed
\begin{align*}
(a(x)^2-x^2)\sqrt{2\pi}e^{f(x)^2/2}T(f(x)) &\leq \sqrt{2\pi}p^2e^{H^2/2}T(H) \\
&\leq -2\sqrt{2\pi}\ln (1-c)e^{H^2/2}T(H) \\
&= \sqrt{\frac{8}{\pi}} < \sqrt{w^2-x^2},
\end{align*}
where the definition of the constant $c$ emerges. Namely, we set
\[ c = 1-\exp\left (-\frac{1}{\pi e^{H^2/2}T(H)}\right ) > 0.64, \]
which completes the proof.
\end{proof}

\begin{remark}\label{remark1}
It might seem unclear why $c > 0.64$. However, it is very easy to verify. Firstly, we check by direct computation that $G(\sqrt{8/\pi}) > F(0.7)$, whence $H > 0.7$ by virtue of the monotonicity of $F$. Secondly, we observe that the dependence $c$ on $H$ is increasing as it was mentioned that $y \longmapsto e^{y^2/2}T(y)$ for $y > 0$ decreases. Thus
\[c = 1-\exp\left (-\frac{1}{\pi e^{0.7^2/2}T(0.7)}\right ) > 0.64.\]
\end{remark}

From the isoperymetric-like inequality \eqref{eq:thm2} proved in the last theorem we have already inferred (cf. steps (I)-(III) presented at the very beginning of this section) that
\[ \nu_n(A) = \nu_n(P) \leq c \quad \textrm{implies} \quad \nu_A'(1) \geq \nu_P'(1). \]
As it was said, this in turn gives the comparison of the measures of $A$ and of a cylinder $P$ when we shrink these sets by dilating them. We can also use this implication in order to predict to some extent what happens with measures when we expand our sets (the simple reasoning which ought to be repeated may be found in \cite{KwaSa})
\begin{cor}\label{cor1}
For any convex rotationally symmetric set $A \subset \C^n$, with measure $\nu_n(A) \leq c$, and a cylinder $P$ satisfying $\nu_n(A) = \nu_n(P)$, we have
\begin{equation}\label{eq:cor1}
\nu_n(tA) \geq \nu_n(tP), \qquad \textrm{for $1 \leq t \leq t_0$},
\end{equation}
where $t_0 \geq 1$ satisfies $\nu_n(t_0A) = c$.
\end{cor}

\section{Some remarks}\label{sec:remarks}

\begin{remark}\label{remark2}
Generally,  without the assumption on the measure of a set $A$ Theorem \ref{thm2} fails. To see this let us consider a cylindrical frustum $A = \{(z,t) \in \R^2\times\R \ | \ |z| \leq w, t \leq y\}$ with the radius $w$ and the height $y$. This is not exactly a set as in the assumptions of Theorem \ref{thm2}, that is, lying under a graph of a smooth concave function (there is a problem with smoothness), but an easy approximation argument will fill in the gap. Take a cylinder $P = \{z \in \R^2 \ | \ |z| \leq p\}\times\R$ with the same measure as $A$, which means that $p$ is taken such that
\[ \Phi(y)(1-e^{-w^2/2}) = \gamma_3(A) = \gamma_3(P) = 1-e^{-p^2/2}.\]
We show that for some large enough $w$ and $y$ there actually holds the reverse inequality to that one stated in Theorem \ref{thm2}
\[ w\gamma_3^+(A) < p\gamma_3^+(P). \]
Indeed, let us fix the parameters of the cylindrical frustum such that
\[ e^{-w^2/2} = T(y), \quad y > 0. \]
Thus $1-e^{-w^2/2} = \Phi(y)$. To simplify some calculations, let us define a function 
\[ g(y) = \frac{1}{\sqrt{2\pi}e^{y^2/2}T(y)}. \]
Now, the relation between $w$ and $y$ may be written as $w^2 = -2\ln T(y) = y^2 + 2\ln \left (\sqrt{2\pi}g(y)\right )$. Furthermore, we have
\[ y < g(y) < \sqrt{y^2+2}, \qquad y>0 ,\]
(the left inequality is a standard estimation for $T(y)$ while the right one follows from Lemma 2, \cite{latole}) so
\begin{align*}
w\gamma_3^+(A) &= w\left (we^{-w^2/2}\Phi(y) + \frac{e^{-y^2/2}}{\sqrt{2\pi}}(1-e^{-w^2/2})\right ) \\
&= T(y)\left (w^2\Phi(y) + w\Phi(y)\frac{e^{-y^2/2}}{\sqrt{2\pi}T(y)}\right ) < T(y)\left (w^2\Phi(y) + wg(y)\right ) \\
&< T(y)\left (w^2\Phi(y) + \sqrt{y^2+2\ln\left (\sqrt{2\pi}g(y)\right )}\sqrt{y^2+2}\right ) \\
&\leq T(y)\left (w^2\Phi(y) + y^2 + \ln\left (\sqrt{2\pi}g(y)\right ) + 1\right ) \\
&= T(y)\left (w^2\left (1+\Phi(y)\right ) + 1 - \ln\left (\sqrt{2\pi}g(y)\right )\right ).
\end{align*}
Let us choose $y$ such that 
\[1 - \ln\left (\sqrt{2\pi}g(y)\right ) < -2(1+\Phi(y))\ln\left (1+\Phi(y)\right ).\]
Then we are able to continue our estimations as follows
\begin{align*}
w\gamma_3^+(A) &< T(y)\left (w^2\left (1+\Phi(y)\right ) -2(1+\Phi(y))\ln\left (1+\Phi(y)\right )\right ) \\
&= -2T(y)\left (1+\Phi(y)\right )\ln \left (e^{-w^2/2}\left (1+\Phi(y)\right )\right ) \\
&= p^2e^{-p^2/2} = p\gamma_3^+(P),
\end{align*}
because
\[ e^{-p^2/2} = 1 - \Phi(y)(1-e^{-w^2/2}) = T(y) + \Phi(y)e^{-w^2/2} = T(y)\left (1+\Phi(y)\right ).\]
\end{remark}

\begin{remark}\label{remark3}
In the previous remark we have seen that the assumption on the measure in Theorem \ref{thm2} is essential. Consequently, the technique which have been used leads from this theorem to Theorem \ref{thm1} also with the restriction on the measure. Nevertheless, we may obtain a weaker version of the inequality \eqref{eq1} dropping the inconvenient assumption that $\gamma_3(A) \leq c$. This result reads as follows
\begin{thm}\label{thm3}
There exists a constant $K = 3$ such that for any convex rotationally symmetric set $A \subset \C^n$ and a cylinder $P$ satisfying $\nu_n(A) = \nu_n(P)$, we have
\begin{equation}\label{eq:thm3}
\nu_n((1+K(t-1))A) \geq \nu_n(tP), \qquad \textrm{for $t \geq 1$}.
\end{equation}
\end{thm}
\begin{proof} Let us denote $\ell (t) = 1+K(t-1)$.

It suffices to prove \eqref{eq:thm3} only for sets with big measure, i.e. $\nu_n(A) \geq c$, where $c$ is the constant from Theorem \ref{thm1}. Indeed, assume that \eqref{eq:thm3} holds for all convex rotationally symmetric sets $A$ such that $\nu_n(A) \geq c$. We are going to show this inequality also for a set $A$ with the measure less than $c$. Let us fix such a set and take $t_0 >1$ such that $\nu_n(t_0A) = c$. From Corollary \ref{cor1} we get
\[ \nu_n(tA) \geq \nu_n(tP), \qquad t \leq t_0. \]
Now, we are to prove \eqref{eq:thm3} for $t > t_0$. Let $Q$ be a cylinder with the same measure as $t_0A$. Applying what we have assumed we obtain
\begin{equation}\label{eq:proofthm3:1}
\nu_n\left (\ell (t) (t_0A)\right ) \geq \nu_n(tQ), \qquad t \geq 1.
\end{equation}
One can make two simple observations
\begin{align*}
\ell (t)t_0 &< \ell(t_0t), \\
\nu_n(Q)=\nu_n(t_0A) \geq \nu_n(t_0P) &\Longrightarrow \nu_n(tQ) \geq \nu_n(tt_0P).
\end{align*}
Together with the inequality \eqref{eq:proofthm3:1} this yields
\[ \nu_n\left (\ell (tt_0)A\right ) \geq \nu_n(tt_0P), \qquad t \geq 1, \]
what is just the desired inequality.

Henceforth, we are going to deal with the proof of inequality \eqref{eq:thm3} in the case of $\nu_n(A) \geq c$. The idea is to exploit the deep result of Lata{\l}a and Oleszkiewicz concerning dilations in the real case. Namely, from Theorem 1  of \cite{latole} we have
\[ \nu_n\left (\ell (t) A\right ) \geq \nu_n\left (\ell (t) S\right ), \qquad t \geq 1, \]
where
\[ S = \{(z_1, \ldots, z_n) \in \C^n \ | \ | \mathrm{Re} z_1 | \leq s\}, \]
is a strip of the width $2s$ chosen so that $\nu_n(A) = \nu_n(S) = 1-2T(s)$. Therefore, we end the proof, providing that we show 
\[ \nu_n\left (\ell (t) S\right ) \geq \nu_n(tP), \qquad t \geq 1.\]
This inequality in turn can be written more explicitly. We have
\[  \nu_n\left (\ell (t) S\right ) = 1 - 2T\left (\ell (t)s\right ),\]
and using the relation $1 - e^{-p^2/2} = \nu_n(P) = \nu_n(A) = \nu_n(S) = 1-2T(s)$ we get $e^{-p^2/2} = 2T(s)$. Hence
\[ \nu_n(tP) = 1-e^{-(tp)^2/2} = 1-\left (2T(s)\right )^{t^2}. \]
Thus it is enough to show that
\begin{equation}\label{eq:proofthm3:2}
\left (2T(s)\right )^{t^2} \geq 2T\left (\ell (t)s\right ), \qquad t \geq 1, \ s \geq s_0,
\end{equation} 
where $s_0$ is such that a strip with the width $2s_0$ has the measure equals to $c$, i.e. $1-2T(s_0) = c$. Since $c > 0.64$, it follows that $T(s_0) < 0.18 < T(0.9)$, so $s_0 > 0.9$.

Let us deal with the inequality \eqref{eq:proofthm3:2}. For $t$ close to $1$ we will apply the Pr\'ekopa-Leindler inequality \cite{Gardner}, Theorem 7.1. To see this, let us fix $s \geq s_0$ and $t \geq 1$ and consider functions
\begin{align*}
f(x) &= \frac{2}{\sqrt{2\pi}}e^{-x^2/2}\I_{[\ell (t) s, \infty)}(x), \\
g(x) &= \frac{2}{\sqrt{2\pi}}e^{-x^2/2}\I_{[0, \infty)}(x), \\
h(x) &= \frac{2}{\sqrt{2\pi}}e^{-x^2/2}\I_{[s, \infty)}(x).
\end{align*}
It is not hard to assert that $f(x)^{1/t^2}g(y)^{1-1/t^2} \leq h\left (\frac{1}{t^2}x + \left (1-\frac{1}{t^2}\right )y\right )$ holds for any $x, y \in \R$ if and only if $\ell(t) s \geq t^2s$, or equivalently $t \leq K-1 = 2$. Then, by virtue of Pr\'ekopa-Leindler inequality, we obtain
\[ \left (2T\left (\ell (t)s\right )\right )^{1/t^2} = \left (\int_\R f\right )^{1/t^2}\left (\int_\R g\right )^{1-1/t^2} \leq \int_\R h = 2T(s). \]

Now we are left with the proof of \eqref{eq:proofthm3:2} in the case of $t > 2$ and $s \geq s_0$. To handle it, we use the asymptotic behaviour of the function $T$ and conduct some tedious calculations. In accordance with the standard estimate from above of the tail probability of the Gaussian distribution we get
\[ T\left (\ell (t)s\right ) <  \frac{1}{\sqrt{2\pi}}\frac{1}{\ell (t) s}e^{-\ell (t)^2s^2/2},\]
whereas from Lemma 2, \cite{latole}
\[ T(s) >  \frac{1}{\sqrt{2\pi}}\frac{1}{\sqrt{s^2+2}}e^{-s^2/2}.\]
Therefore, in order to show \eqref{eq:proofthm3:2} it is enough to prove
\[ \left (\frac{2}{\sqrt{2\pi}} \right )^{t^2}\frac{1}{(s^2+2)^{t^2/2}}e^{-t^2s^2/2} \geq \frac{2}{\sqrt{2\pi}}\frac{1}{\ell (t) s}e^{-\ell (t)^2s^2/2},\]
which is equivalent to the inequality
\[ \exp\left (\frac{s^2}{2}\left (\ell (t)^2 - t^2\right )\right ) \geq \left (\sqrt{\frac{\pi}{2}}\right )^{t^2-1}\frac{(s^2+2)^{t^2/2}}{\ell(t) s}, \qquad s \geq s_0, \ t \geq 2.\]
Taking the logarithm of both sides, putting the definition of $\ell (t) = 1 + K(t-1) = 3t - 2$ and simplifying we have to prove
\[ \left (8s^2 - \ln \left (\frac{\pi}{2}\left (s^2+2\right )\right )\right )t^2 - 12s^2t + 4s^2 + \ln \left (\frac{\pi}{2}s^2\right ) + 2\ln\left (3t-2\right ) \geq 0. \]
Let us call the left hand side by $F(s, t)$. Notice that
\begin{align*}
\frac{\partial F}{\partial t}(s, t) &= 2\left (8s^2 - \ln \left (\frac{\pi}{2}\left (s^2+2\right )\right )\right )t - 12s^2 + \frac{2}{3t-2} \\
&> 2\left (5s^2 - \ln \left (\frac{\pi}{2}\left (s^2+2\right )\right )\right )t > 2\left (5s^2 - \frac{\pi}{2e}(s^2+2)\right )t \\
&= 2\left (\left (5-\frac{\pi}{2e}\right )s^2 - \frac{\pi}{e}\right )t \geq 2\left (\left (5-\frac{\pi}{2e}\right )s_0^2 - \frac{\pi}{e}\right )t \\
&> 2\left (\left (5-\frac{\pi}{2e}\right )\cdot 0.81 - \frac{\pi}{e}\right )t > 0,
\end{align*}
where in the first inequality we used only the assumption that $t>2$ getting $-12s^2 > -6ts^2$ and neglected the term $\frac{2}{3t-2}$ as being positive, while in the second one we evoked well-known inequality $\ln x \leq \frac{x}{e}$. Knowing that this derivative is positive, we will finish if we check that $F(s, 2) > 0$. However, it can be done by direct computation
\begin{align*}
F(s,2) &= 4\left (8s^2 - \ln \left (\frac{\pi}{2}\left (s^2+2\right )\right )\right ) - 24s^2 + 4s^2 + \ln s^2 + \ln \frac{\pi}{2} + 2\ln 4 \\
&= 4\left (3s^2 - \ln \left (\frac{\pi}{2}\left (s^2+2\right )\right )\right ) + \ln \left (8\pi s^2\right ) \\
&> 4\left (\left (3-\frac{\pi}{2e}\right )s^2 - \frac{\pi}{e}\right ) > 0.
\end{align*}
The proof is now complete.
\end{proof}
\end{remark}

\section{Technical lemmas}\label{sec:lemmas}

We are going to prove some rather technical lemmas which have helped us with the proof of  Lemma \ref{lm3}.

\begin{lm}\label{lm1}
The function $F$, defined in \eqref{eq:defF}, is increasing and onto $(0, \infty)$.
\end{lm}
\begin{proof}
As far as the monotonicity is concerned, it suffices to prove that $F$ is nondecreasing. Indeed, if $F$ were be constant on some interval, it would be constant everywhere as $F$ is an analytic function.

Clearly, $F$ is nondecreasing iff $1/F$ is nonincreasing. Notice that
\begin{align*}
\frac{1}{F(y)} &= \frac{-e^{-y^2/2}}{\sqrt{2\pi}}\frac{1}{T(y)\ln T(y)} = \frac{T'(y)}{T(y)\ln T(y)} \\
&= \frac{\left (-\ln T(y)\right )'}{-\ln T(y)} = \left (\ln \left (-\ln T(y)\right )\right )',
\end{align*}
thus $1/F$ is nonincreasing iff $y \longmapsto \ln \left (-\ln T(y)\right )$ is concave, that is for any $x, y \in \R$, $\lambda \in (0, 1)$
\[ -\ln T(\lambda x + (1-\lambda)y) \geq \left (-\ln T(x)\right )^\lambda\left (-\ln T(y)\right )^{1-\lambda}. \]
Since $\lim_{x\to -\infty} (-\ln T(x)) = 0$, we have
\[ -\ln T(x) = \int_{-\infty}^x \left (-\ln T(t)\right )'\dd t = \int_{-\infty}^x \frac{e^{-t^2/2}}{\sqrt{2\pi}T(t)} \dd t, \]
and the above inequality will hold by virtue of the Pr\'ekopa-Leindler inequality providing that we check the assumption. In our case it reduces to verify whether the function $\ln \frac{e^{-t^2/2}}{\sqrt{2\pi}T(t)}$ is concave. Calculating the second derivative one can easily check that it is non-positive iff
\begin{align*}
0 &\geq T(t)^2 + \frac{e^{-t^2/2}}{\sqrt{2\pi}}tT(t) - \left (\frac{e^{-t^2/2}}{\sqrt{2\pi}}\right )^2 \\
&= \left (T(t) - \frac{e^{-t^2/2}}{\sqrt{2\pi}}\frac{\sqrt{t^2+4}-t}{2}\right )\left (T(t) + \frac{e^{-t^2/2}}{\sqrt{2\pi}}\frac{\sqrt{t^2+4}+t}{2}\right ), \quad t \in \R,
\end{align*}
which is equivalent to
\[ T(t) \geq \frac{e^{-t^2/2}}{\sqrt{2\pi}}\frac{\sqrt{t^2+4}-t}{2}, \qquad t \in \R. \]
For $t \geq 0$ it follows from a well-known Komatsu's estimate (cf. \cite{itomckean}, page 17). For $t < 0$ we have $T(t) > 1/2$, hence 
\[2T(t)\sqrt{2\pi}e^{t^2/2}+t \geq \sqrt{2\pi}(1+t^2/2) + t > 0,\]
and
\begin{align*}
\left (2T(t)\sqrt{2\pi}e^{t^2/2} + t\right )^2 &> \left (\sqrt{2\pi}(1+t^2/2) + t\right )^2 \\
&= 2\pi \left (1+\frac{t^2}{2}\right )^2 + 2\sqrt{2\pi}\left (1+\frac{t^2}{2}\right )t + t^2 \\
&= 2\left (1+\frac{t^2}{2}\right )\left (\frac{\pi}{2}t^2 + \sqrt{2\pi}t + \pi\right ) + t^2 \\
&> 2\left (\left (\sqrt{\frac{\pi}{2}}t + 1\right )^2 + \pi - 1\right ) + t^2 \\
&> 2(\pi - 1) + t^2 > t^2 + 4.
\end{align*}
This completes the proof of the monotonicity of $F$. 

$F$ is onto $(0, \infty)$ as
\begin{align*}
F(y) &\xrightarrow[y\to -\infty]{} 0, \\
F(y) &\xrightarrow[y\to +\infty]{} \infty.
\end{align*}
\end{proof}

\begin{lm}\label{lm2}
The function $G$, defined in \eqref{eq:defG}, is increasing for $u \geq \sqrt{8/\pi}.$
\end{lm}
\begin{proof}
We have
\[ G'(u) = \frac{1-e^{-u^2/2}-u^2e^{-u^2/2}}{2\left (1-e^{-u^2/2}\right )^2}, \]
so $G'(u) > 0$ iff $e^{u^2/2} > 1+u^2$. It is true for $u^2 > 8/\pi$ since $e^{4/\pi} > 1 + 8/\pi$.
\end{proof}

\begin{proof}[Proof of Lemma \ref{lm3}]
\emph{(i)} Using the convexity of the function $-\ln$ we get
\begin{align*}
& -2\sqrt{2\pi}e^{y^2/2}T(y)\ln\left (T(y) + \Phi(y)e^{-u^2/2}\right ) \\
&\leq 2\sqrt{2\pi}e^{y^2/2}T(y)\left (-T(y)\ln 1 - \Phi(y)\ln e^{-u^2/2}\right ) \\
&= \sqrt{2\pi}e^{y^2/2}T(y)\Phi (y) u^2 \leq \sqrt{\frac{\pi}{8}}u^2 \leq u,
\end{align*}
where we use $\sup_{y \in \R} \sqrt{2\pi}e^{y^2/2}T(y)\Phi (y) = \sqrt{\frac{\pi}{8}}$ (see \cite{latole}, Lemma 5).

\emph{(ii)} Since $T(y) + \Phi (y)e^{-u^2/2} =  e^{-u^2/2} + (1-e^{-u^2/2})T(y)$, we may also apply the convexity of $-\ln$ to points $1$, $T(y)$ with weights $e^{-u^2/2}$, $1-e^{-u^2/2}$ and obtain
\begin{align*}
& -2\sqrt{2\pi}e^{y^2/2}T(y)\ln\left (T(y) + \Phi (y)e^{-u^2/2}\right ) \\
&\leq -2\sqrt{2\pi}e^{y^2/2}T(y)\ln T(y) (1-e^{-u^2/2}) \\
&= \frac{F(y)}{G(u)}u \leq \frac{F(H)}{G \left (\sqrt{8/\pi}\right )}u = u.
\end{align*}
\end{proof}

\section*{Acknowledgments}

The author would like to thank Prof. R. Lata{\l}a for introducing him into the subject as well as for many helpful comments and among them, especially for the counterexample stated in Remark \ref{remark2}.

\end{document}